\documentclass{article}
\usepackage{amsmath,amsthm,amscd,amssymb}
\usepackage[mathscr]{eucal}
\usepackage{amssymb}
\usepackage{latexsym}

\theoremstyle{definition}
\newtheorem{Def}{Definition}[section]
\newtheorem{Thm}[Def]{Theorem}
\newtheorem{Prop}[Def]{Proposition}
\newtheorem{Rem}[Def]{Remark}
\newtheorem{Ex}[Def]{Example}

\newtheorem{Lem}[Def]{Lemma}

\numberwithin{equation}{section}

\title{Theta operator for Hermitian modular forms over the Eisenstein field}

\author{Shoyu Nagaoka and Sho Takemori}

\date{}
\begin{document}

\maketitle

\begin{abstract}
In this paper, we have investigated the 
mod $p$ kernel of the theta operator for Hermitian modular forms
when the base field is the Eisenstein field.

\end{abstract}

\section{Introduction}
\label{intro}
The first attempt to generalize Ramanujan's theta operator to a
higher degree case was made in \cite{B-N}. 
Subsequently, this generalization was developed by several researchers
(e.g., \cite{K-K-N}, \cite{B-K-N},\cite{K-N}).
\\
 In the case of Siegel modular forms, the theta operator $\varTheta$
is defined by $F=\sum a_F(T)q^T$$\longrightarrow$
$\varTheta (F):=\sum\text{det}(T)\cdot a_F(T)q^T$ for the generalized $q$-expansion
$F=\sum a_F(T)q^T$.
If we use the theta operator, some congruence property that satisfies Igusa's Siegel cusp
form $\chi_{35}$ can be expressed as
$$
(*)\qquad\qquad\qquad\qquad \varTheta (\chi_{35}) \equiv 0 \pmod{23}\quad
(\text{cf. \cite{K-K-N}}).
$$
For a modular form $F$, the image $\varTheta (F)$ is not necessarily a modular form
in general. However, we know that, for a fixed prime number $p$, $\varTheta (F)$
is congruent to a {\it true} modular form $G$ mod $p$ 
under some condition on $p$ (cf. \cite{B-N}) as follows:
$$
\varTheta (F) \equiv G \pmod{p},
$$
The modular form $G$ occasionally becomes zero,
that is, $\varTheta (F) \equiv 0 \pmod{p}$. In this case, we say that $F$ is
an element of the mod $p$ kernel of the theta operator.\\
In the case of degree 2 Siegel modular forms, we have a number of examples of
such forms. For example, Igusa's cusp form $\chi _{35}$ is an element of the mod $23$
kernel of the theta operator as stated above. Moreover, the Siegel Eisenstein
series $E_{12}^{(2)}$ and the Siegel theta series 
$\vartheta_{\mathcal{L}}^{(2)}$ associated
with the Leech lattice $\mathcal{L}$ satisfy
$$
(\dagger)\qquad\qquad\qquad\qquad\qquad\varTheta (E_{12}^{(2)}) \equiv \varTheta (\vartheta_{\mathcal{L}}^{(2)})
\equiv 0 \pmod{23},\quad (\text{ e.g. cf. \cite{N-T1}}).
$$
In this paper, we shall show that similar phenomena exist in the case of
Hermitian modular forms of degree 2 over the Eisenstein field $\mathbb{Q}(\sqrt{3}\,i)$.\\
In \cite{D-K}, Dern and Krieg determined the structure of the graded rings of degree 2 
Hermitian modular forms in the cases $\mathbb{Q}(i)$ (the Gaussian field) and 
$\mathbb{Q}(\sqrt{3}\,i)$ (the Eisenstein field). In the case of $\mathbb{Q}(\sqrt{3}\,i)$,
they showed that there are two odd weight forms $\phi_9$ and $\phi_{45}$ in the set
of generators. In this study, we shall show that
$$
(**)\qquad\qquad \varTheta (\phi_9) \equiv 0 \pmod{2},\qquad
                       \varTheta(\phi_{45}) \equiv 0 \pmod{11},
$$
in $\S$ \ref{sec:4-2}.\\ 
The congruence relations $(*)$ and $(**)$ lead us to the following conjecture:
\vspace{1mm}
\\
\quad {\it Any odd weight modular form with Borcherds product will be in the mod $p$ kernel
of the theta operator for a suitable prime number $p$.}
\vspace{1mm}
\\
(Igusa's cusp form $\chi_{35}$ is a typical example of Borcherds product 
(cf. \cite{G-N}, Theorem 1.5.). Moreover, the modular forms $\phi_9$ and $\phi_{45}$
are constructed as Borcherds products (\cite{D-K}, Corollary 3).)
\vspace{2mm}
\\
The isometry classes of rank 12 Eisenstein lattices were classfied by 
Hentschel, Krieg and Nebe in \cite{H-K-N}. They showed that there are exactly five
isometry classes. According to their notation, we write the corresponding
representative Gram matrices as $H_i$$(i=1,\ldots,5)$. 
We show that the weight 12 Hermitian Eisenstein series
$E_{12,\mathbb{Q}(\sqrt{3}\,i)}^{(2)}$ and the
Hermitian theta series $\vartheta_{H_i}^{(2)}$ $(i=4,5)$ satisfy
$$
(\dagger\dagger)\qquad\qquad\qquad
\varTheta(E_{12,\mathbb{Q}(\sqrt{3}\,i)}^{(2)}) \equiv 
\varTheta (\vartheta_{H_4}^{(2)}) \equiv \varTheta (\vartheta_{H_5}^{(2)})
\equiv 0 \pmod{11},
$$
where $H_5$ corresponds to the Hermitian Leech lattice
(cf. $\S$ \ref{sec:3-3} ) .
\\
The congruence relation $(\dagger\dagger)$ and the corresponding
result in the case of the Gaussian field (Theorem 8 in \cite{K-N})
leads to the second conjecture:
\vspace{1mm}
\\
\quad {\it Any Hermitian theta series associated with the Hermitian Leech
lattice will be in the mod $p$ kernel of the theta operator for a suitable prime 
number $p$.}
\vspace{1mm}
\\
\section{Preliminaries}
\label{sec:2}
\subsection{Notation}
\label{sec:2-1}
The {\it Hermitian upper half-space} of degree 2 is given as
$$
\mathbb{H}_2:=\{ Z\in M_2(\mathbb{C})\;\mid\;
\frac{1}{2i}(Z-{}^t\overline{Z})>0\;\},
$$
where ${}^t\overline{Z}$ denote the transpose, and the complex conjugate
of $Z$. Let $\boldsymbol{K}$ be an imaginary quadratic number field with the
discriminant $d_{\boldsymbol{K}}$ and the ring of integers
$$
\mathcal{O}_{\boldsymbol{K}}=\mathbb{Z}+\mathbb{Z}\omega,\quad
\omega=
\begin{cases}
i\sqrt{|d_{\boldsymbol{K}}|}/2,  & \text{if $d_{\boldsymbol{K}} \equiv 0 \pmod{4}$},\\
(1+i\sqrt{|d_{\boldsymbol{K}}|})/2 & \text{if $d_{\boldsymbol{K}} \not\equiv 0 \pmod{4}$}.
\end{cases}
$$
Then
$$
\Gamma_2(\mathcal{O}_{\boldsymbol{K}}):=
\{\,M\in M_4(\mathcal{O}_{\boldsymbol{K}})\ \mid\;MJ{}^t\overline{M}=J\,\},\quad
J:=\begin{pmatrix}  0 & -E_2\\ E_2 & 0 \end{pmatrix},
$$
is called the {\it Hermitian modular group} of degree 2 over $\boldsymbol{K}$.
The group $\Gamma_2(\mathcal{O}_{\boldsymbol{K}})$ acts on $\mathbb{H}_2$
by the fractional linear transformation $Z\longmapsto M\langle Z\rangle$
$:=(AZ+B)(CZ+D)^{-1}$, $M=\binom{A\,B}{C\,D}$.
If $\Gamma\subset\Gamma_2(\mathcal{O}_{\boldsymbol{K}})$ is a subgroup of
finite index and $\nu$ is an abelian character of $\Gamma$, the space
$M_k(\Gamma,\nu)$,\,$k\in\mathbb{Z}$ of the {\it Hermitian modular forms} of
weight $k$ and character $\nu$ with respect to $\Gamma$ consists of
all the holomorphic functions $F:\,\mathbb{H}_2\longrightarrow \mathbb{C}$ that
satisfy
$$
F\mid_kM(Z):=
\text{det}(CZ+D)^{-k}F(M\langle Z\rangle)=\nu (M)\cdot F(Z)
$$
for all $M=\binom{AB}{CD}\in\Gamma$. The subspace $S_k(\Gamma,\nu)$
of the cusp forms is characterized by the condition
$$
F\Big{|}_k\begin{pmatrix} {}^t\overline{U} & 0 \\ 0 & U^{-1}\end{pmatrix}
\Big{|}\Phi \equiv 0\quad
\text{for all}\quad U\in GL_2(\boldsymbol{K}),
$$
where $\Phi$ is the Siegel $\Phi$-operator. There is an exceptional
automorphism of the Hermitian upper half-space
$$
I_{tr}\,:\,\mathbb{H}_2\longrightarrow \mathbb{H}_2,\quad
Z\longmapsto {}^tZ,
$$
satisfying
$$
M\circ I_{tr}=I_{tr}\circ \overline{M}\quad
\text{for all}\quad M\in \Gamma_2(\mathcal{O}_{\boldsymbol{K}}).
$$
The superscript {\it sym} (resp. {\it skew}) denotes the subspace
of the symmetric (resp. skew-symmetric) Hermitian modular forms characterized by
$$
F\circ I_{tr}=F\qquad \text{resp.}\qquad F\circ I_{tr}=-F.
$$
Each $F\in M_k(\Gamma_2(\mathcal{O}_{\boldsymbol{K}}),\text{det}^l)$ possesses
a Fourier expansion of the form
$$
F(Z)=\sum_{0\leq H\in\Lambda_2(\mathcal{O}_{\boldsymbol{K}})}
       a(F;H)\text{exp}(2\pi i\text{tr}(HZ)),
$$
where
$$
\Lambda_2(\mathcal{O}_{\boldsymbol{K}}):=\left\{ H=
                           \begin{pmatrix} m & t \\ \overline{t} & n\end{pmatrix}
                \;\Big{|}\; m,\,n\in\mathbb{Z}_{\geq 0},
                \,t\in\frac{1}{\sqrt{d_{\boldsymbol{K}}}}\mathcal{O}_{\boldsymbol{K}}
                =\mathcal{O}_{\boldsymbol{K}}^{\sharp}\;\right\}.
$$
For simplicity, we write the above Fourier expansion by $F=\sum a(F;H)q^H$.
We may regard that the right-hand side is an element of some formal
power series ring $\mathbb{C}[\![ q]\!]$ (cf. \cite{M-N}, p.248).
\\
Let $R$ be a subring of $\mathbb{C}$. We denote by
$M_k(\Gamma_2(\mathcal{O}_{\boldsymbol{K}}),\text{det}^l)_R$ the $R$-module
consisting of $F\in M_k(\Gamma_2(\mathcal{O}_{\boldsymbol{K}}),\text{det}^l)$
all of whose Fourier coefficients $a(F;H)$ are in $R$. In this case, we may regard
it as an element of $R[\![ q]\!]$.
\vspace{1mm}
\\
For a prime number $p$, we denote by $\mathbb{Z}_{(p)}$ the local ring at
$p$ (i.e. the ring of $p$-integral rational numbers). For two elements
$F_i=\sum a(F_i;H)q^H\in \mathbb{Z}_{(p)}[\![q]\!]$\,($i=1,2$), we write
$$
F_1 \equiv F_2 \pmod{p}
$$
when $a(F_1;H) \equiv a(F_2;H) \pmod{p}$ holds for all 
$H\in\Lambda_2(\mathcal{O}_{\boldsymbol{K}})$.
\subsection{Eisenstein series and theta series}
\label{sec:2-2}
Examples of the Hermitian modular forms are given by the Hermitian Eisenstein
series and the Hermitian theta series.\\
The {\it Hermitian Eisenstein series} of weight $k$ and degree 2 over $\boldsymbol{K}$
is defined by
$$
E_{k,\boldsymbol{K}}^{(2)}(Z)
:=\sum_{M=\binom{*\,*}{C\,D}:\left\{\binom{*\,*}{0\,*}\right\}\backslash \Gamma_2(\mathcal{O}_{\boldsymbol{K}})}
(\text{det}(M))^{k/2}\text{det}(CZ+D)^{-k},
$$ 
where $k>4$ is an even integer. It is known that
$E_{k,\boldsymbol{K}}^{(2)}\in$ 
$M_k(\Gamma_2(\mathcal{O}_{\boldsymbol{K}}),\text{det}^{-k/2})_{\mathbb{Q}}^{sym}$
(cf. \cite{D-K}). Moreover
$E_{4,\boldsymbol{K}}^{(2)}\in$ 
$M_4(\Gamma_2(\mathcal{O}_{\boldsymbol{K}}),\text{det}^{-2})_{\mathbb{Q}}^{sym}$
is defined as the Maa\ss\, lift (cf. \cite{D-K}).
An explicit formula for $a(E_{k,\boldsymbol{K}}^{(2)};H)$ was given by Krieg \cite{K}
in the case when the class number of $\boldsymbol{K}$ was one.
\\
The second example is the theta series. Let $S\in Her_m(\boldsymbol{K})$
be a positive definite even matrix with respect to $\mathcal{O}_{\boldsymbol{K}}$.
The degree 2 {\it Hermitian theta series} associated with $S$ is defined by
$$
\vartheta^{(2)}(Z;S):=\sum_{G\in M_{m,2}(\mathcal{O}_{\boldsymbol{K}})}
\text{exp}(\pi i\text{tr}(Z{}^t\overline{G}SG)).
$$
If $S$ satisfies an additional condition $\text{det}(S)=(2/\sqrt{d_{\boldsymbol{K}}})^m$
(i.e. $S$ is unimodular),
then $m \equiv 0 \pmod{4}$ and
$$
\vartheta^{(2)}(Z;S)\in 
M_m(SL_4(\boldsymbol{K})\cap\Gamma_2(\mathcal{O}_{\boldsymbol{K}}),1)_{\mathbb{Z}}^{sym},
$$
(cf. \cite{C-R}).
\section{Hermitian modular forms over the Eisenstein field}
\label{Eisenstein}
In the rest of this paper, we treat the case in which
$$
\boldsymbol{K}=\mathbb{Q}(\sqrt{3}\,i)\quad (\text{Eisenstein field}).
$$
In this case
$$
d_{\boldsymbol{K}}=-3,\qquad 
\mathcal{O}_{\boldsymbol{K}}=\mathbb{Z}+\omega\mathbb{Z},\quad
\omega=(1+\sqrt{3}\,i)/2.
$$
It is known that
\begin{equation}
\label{SL}
M_k(\Gamma_2(\mathcal{O}_{\boldsymbol{K}}),\text{det}^l)
=
\begin{cases}
M_k(SL_4(\boldsymbol{K})\cap\Gamma_2(\mathcal{O}_{\boldsymbol{K}}),1)
& \text{if $k \equiv l \pmod{3}$},\\
\{0\} & \text{otherwise}.
\end{cases}
\end{equation}

In this case, the formal power series $\mathbb{C}[\![ q]\!]$ introduced in $\S$ \ref{sec:2-1}
is explicitly given as follows:
\begin{equation}
\label{power}
\mathbb{C}[\![ q]\!]:=
\mathbb{C}[\dot{q}_{12}^{\pm 1},\ddot{q}_{12}^{\pm 1}][\![\dot{q}_{11},\dot{q}_{22}]\!],
\end{equation}
where $\dot{q}_{12}:=\text{exp}(2\pi i(z_{12}-z_{21})/2\sqrt{3}i)$,
$\ddot{q}_{12}:=\text{exp}(2\pi i(z_{12}+z_{21})/2)$,
$\dot{q}_{11}:=\text{exp}(2\pi iz_{11})$, and $\dot{q}_{22}:=\text{exp}(2\pi iz_{22})$
for $Z=\begin{pmatrix}z_{11}&z_{12}\\z_{21}&z_{22} \end{pmatrix}$.
\subsection{Structure of the graded ring}
\label{sec:3-1}
In the case $\boldsymbol{K}=\mathbb{Q}(\sqrt{3}\,i)$, Dern-Krieg \cite{D-K}
determined the structure of the graded ring
$$
\underset{k\in\mathbb{Z}}{\bigoplus}\,M_k(\Gamma_2(\mathcal{O}_{\boldsymbol{K}}),\text{det}^k)
=\underset{k\in\mathbb{Z}}{\bigoplus}\,M_k(SL_4(\boldsymbol{K})\cap
\Gamma_2(\mathcal{O}_{\boldsymbol{K}}),1)\quad (\text{cf. (\ref{SL})}).
$$
\begin{Prop}
\label{odd}
{\rm (\cite{D-K}, Corollary 3)} Let $\boldsymbol{K}=\mathbb{Q}(\sqrt{3}\,i)$.
Then there exist the Borcherds products
$$
\phi_9\in S_9(\Gamma_2(\mathcal{O}_{\boldsymbol{K}}),1)^{skew}\quad
\text{and}\quad
\phi_{45}\in S_{45}(\Gamma_2(\mathcal{O}_{\boldsymbol{K}}),1)^{sym}.
$$
\end{Prop}
We will study the mod $p$ properties of $\phi_9$ and $\phi_{45}$ in $\S$ \ref{sec:4-2}.
\begin{Thm}
\label{graded}
{\rm (\cite{D-K}, Theorem 6)} 
Let $\boldsymbol{K}=\mathbb{Q}(\sqrt{3}\,i)$.\\
(1)\,The graded ring
$$
\underset{k\in\mathbb{Z}}{\bigoplus}\,M_k(\Gamma_2(\mathcal{O}_{\boldsymbol{K}}),\text{det}^k)
=\underset{k\in\mathbb{Z}}{\bigoplus}\,M_k(SL_4(\boldsymbol{K})\cap
\Gamma_2(\mathcal{O}_{\boldsymbol{K}}),1)
$$
is generated by
$$
E_{4,\boldsymbol{K}}^{(2)},\;\;
E_{6,\boldsymbol{K}}^{(2)},\;\;
\phi_9,\;\;
E_{10,\boldsymbol{K}}^{(2)},\;\;
E_{12,\boldsymbol{K}}^{(2)},\;\;\text{and}\;\;
\phi_{45},
$$
where $\phi_9$ and $\phi_{45}$ are cusp forms given in Proposition \ref{odd}.\\
(2)\, The ideal of cusp forms in $\oplus M_k(\Gamma_2(\mathcal{O}_{\boldsymbol{K}}),\text{det}^k)$ is generated by
$$
\phi_9,\;\;
f_{10},\;\;
f_{12},\;\;\text{and}\;\;
\phi_{45},
$$
where $f_k$ is defined by
\begin{align*}
& f_{10}:=E_{10,\boldsymbol{K}}^{(2)}-E_{4,\boldsymbol{K}}^{(2)}\cdot E_{6,\boldsymbol{K}}^{(2)}
           \in S_{10}(\Gamma_2(\mathcal{O}_{\boldsymbol{K}}),\text{det}^{-5})^{sym},\\
& f_{12}:=E_{12,\boldsymbol{K}}^{(2)}-\frac{441}{691}(E_{4,\boldsymbol{K}}^{(2)})^3
            -\frac{250}{691}(E_{6,\boldsymbol{K}}^{(2)})^2
           \in S_{12}(\Gamma_2(\mathcal{O}_{\boldsymbol{K}}),1)^{sym}.
\end{align*}
\end{Thm}
\subsection{Theta operator on Hermitian modular forms}
\label{sec:3-2}
The theory of a theta operator on the Hermitian modular forms was developed
by several researchers (e.g., \cite{M-S}, \cite{K-N}).\\
We recall that the Fourier expansion of the Hermitian modular form can be regarded
as an element of the formal power series ring $\mathbb{C}[\![ q]\!]$ (cf. (\ref{power})).
The {\it theta operator} over $\mathbb{C}[\![ q]\!]$ is defined by
$$
\varTheta :\,F=\sum a(F;H)q^H\longmapsto 
\varTheta (F):=\sum \text{det}(H)\cdot a(F;H)q^H.
$$
As stated in the introduction, $\varTheta (F)$ is not necessarily a Hermitian modular
form even if $F$ is. However we have the following result.
\begin{Thm}
\label{B-N-Hvarsion}
Let $\boldsymbol{K}=\mathbb{Q}(\sqrt{3}\,i)$ and $p$ be a prime number such
that $p\geq 5$. For any 
$F\in M_k(\Gamma_2(\mathcal{O}_{\boldsymbol{K}}),\text{det}^k)_{\mathbb{Z}_{(p)}}$,
there is a cusp form
$$
G\in S_{k+p+1}
(\Gamma_2(\mathcal{O}_{\boldsymbol{K}}),\text{det}^{k+p+1})_{\mathbb{Z}_{(p)}}
$$
such that
\begin{equation}
\label{equivalence}
\varTheta (F) \equiv G \pmod{p}.
\end{equation}
\end{Thm}
\begin{proof}
The same type of statement in the case $\boldsymbol{K}=\mathbb{Q}(i)$ was
given in \cite{K-N}, Theorem 3. A similar method using the Rankin-Cohen bracket is
applicable for the case $\boldsymbol{K}=\mathbb{Q}(\sqrt{3}\,i)$ (e.g., \cite{M-S}).
\end{proof}
\begin{Ex}
\label{f12}
Here we give an example of (\ref{equivalence}) in the case that 
$F=E_{4,\boldsymbol{K}}^{(2)}$ and $p=7$. \\
Let $f_{12}$ be the cusp form introduced in Theorem \ref{graded}, (2).
We normalize $f_{12}$ as
$$
\tilde{f}_{12}:=-\frac{691\cdot 1847}{2^{13}\cdot 3^6\cdot 5^3\cdot 7^2}\,f_{12},
\qquad
a\left( \tilde{f}_{12},\begin{pmatrix}1 & 1/\sqrt{3}\,i\\ -1/\sqrt{3}\,i & 1     \end{pmatrix}\right)=1.
$$
Then all of the Fourier coefficients of $\tilde{f}_{12}$ are rational integers and
$$
\varTheta (E_{4,\boldsymbol{K}}^{(2)}) \equiv \tilde{f}_{12} \pmod{7}.
$$
For example
$$
\begin{cases}
& a\left(\varTheta(E_{4,\boldsymbol{K}}^{(2)});\begin{pmatrix}1&x\\ \overline{x}&1\end{pmatrix} \right)=4320,\\
& a\left(\tilde{f}_{12};\begin{pmatrix}1&x\\\overline{x}&1\end{pmatrix}\right)  =1,
\end{cases}
$$
for $x\in\boldsymbol{K}$ with $N(x)=1/3$, and
$$
\begin{cases}
& a\left(\varTheta(E_{4,\boldsymbol{K}}^{(2)});\begin{pmatrix}1&0\\ 0&1\end{pmatrix} \right)=17280,\\
& a\left(\tilde{f}_{12};\begin{pmatrix}1&0\\ 0&1\end{pmatrix}\right)=18.
\end{cases}
$$
\end{Ex}
In the congruence $\varTheta (F) \equiv G \pmod{p}$, the modular form $G$
sometimes vanishes identically, that is, $\varTheta (F) \equiv 0 \pmod{p}$.
In this case, $F$ is called an element of the {\it mod p kernel} of the theta
operator.

The main purpose of this paper is to construct such modular forms.

\subsection{Hermitian theta series for Eisenstein lattices}
\label{sec:3-3}
We still assume that $\boldsymbol{K}=\mathbb{Q}(\sqrt{3}\,i)$.
We recall the definition of Eisenstein lattice.\\
The lattice $\Lambda\subset \boldsymbol{K}^r$ is called an {\it Eisenstein lattice}
of rank $r$ if there exist linearly independent vectors
$b_1,\ldots,b_r\in\boldsymbol{K}^r$ such that
\vspace{2mm}
\\
(i)\; $\Lambda=\mathcal{O}_{\boldsymbol{K}}b_1+\cdots+\mathcal{O}_{\boldsymbol{K}}b_r$,
\vspace{1mm}
\\
(ii)\; $\text{det}(\langle b_j,b_k\rangle)=(2/\sqrt{3}\,i)^r$,
\vspace{1mm}
\\
(iii)\; $\langle \lambda,\lambda\rangle\in 2\mathbb{Z}$ for all $\lambda\in\Lambda$,
\vspace{2mm}
\\
where $\langle\,,\,\rangle:\boldsymbol{K}^r\times\boldsymbol{K}^r\longrightarrow \boldsymbol{K}$
is the standard Hermitian scalar product defined by $\langle x,y\rangle={}^t\overline{x}y$.
\\
It is known that Eisenstein lattices exist only if $r$ is multiple of $4$ (cf. $\S$ \ref{sec:2-2}).
Moreover, we have
$$
\vartheta^{(2)}(Z,S)
\in M_r(SL_4(\boldsymbol{K})\cap\Gamma_2(\mathcal{O}_{\boldsymbol{K}}),1)_{\mathbb{Z}}^{sym}
    =M_r(\Gamma_2(\mathcal{O}_{\boldsymbol{K}}),\text{det}^{-r/2})_{\mathbb{Z}}^{sym},
$$
where $S$ is the Gram matrix of an Eisenstein lattice (cf. \cite{H-K-N}, Proposition 2).
\\
Hentschel-Krieg-Nebe classified isometry classes of Eisenstein lattices of rank 12.
\begin{Thm}
\label{hermLeech}
{\rm (\cite{H-K-N}, Theorem 2).}
There are exactly five isometry classes of Eisenstein
lattices rank 12 whose root lattices are $3E_8$, $4E_6$, $6D_4$, $12A_{12}$,
$\emptyset$ (i.e., the Hermitian Leech lattice).
\end{Thm} 
According to \cite{H-K-N}, we write the corresponding Gram matrices as
$H_1,\ldots,H_5$ and consider the theta series
$\vartheta^{(2)}(Z,H_i)$$\in M_{12}(\Gamma_2(\mathcal{O}_{\boldsymbol{K}}),1)_{\mathbb{Z}}^{sym}$.
\\
The following identity is a special case of the analytic version of Siegel's main theorem:
\begin{Prop}
\label{siegel}
\begin{align*}
E_{12,\boldsymbol{K}}^{(2)}=&-\frac{3\cdot 7\cdot 11\cdot 13}{691\cdot 809\cdot 1847}
                                   \vartheta^{(2)}(Z,H_1)+
                                   \frac{2^6\cdot 5^3\cdot 7\cdot 11\cdot 13}
                                   {691\cdot 809\cdot 1847}
                                   \vartheta^{(2)}(Z,H_2)\\
                                   &+\frac{2\cdot 3^8\cdot 5^2\cdot 7\cdot 11\cdot 13}
                                   {691\cdot 809\cdot1847}
                                   \vartheta^{(2)}(Z,H_3)+
                                   \frac{2^{15}\cdot 3^2\cdot 5^2\cdot 7\cdot 13}
                                   {691\cdot 809\cdot 1847}
                                   \vartheta^{(2)}(Z,H_4)\\
                                   &+\frac{2^8\cdot 3^9\cdot 5}
                                   {691\cdot 809 \cdot 1847}
                                   \vartheta^{(2)}(Z,H_5).                                
\end{align*}
\end{Prop}
\begin{proof}
The identity is obtained by the direct calculation of the Fourier coefficients of
$E_{12,\boldsymbol{K}}^{(2)}$ and $\vartheta^{(2)}(Z,H_i)$.
\end{proof}
\subsection{Sturm bound for Hermitian modular forms over the Eienstein field}
\label{sec:3-4}
To prove congruences among Hermitian modular forms, we recall some results
of the Sturm bound in the case of Hermitian modular forms over the Eisenstein
field.
\begin{Thm}
({\rm \cite{K-N}, Theorem 2}).
\label{sturm}
Let $\boldsymbol{K}=\mathbb{Q}(\sqrt{3}\,i)$ and $p$ be a prime number such
that $p\geq 5$. If a Hermitian modular form
$F\in M_k(\Gamma_2(\mathcal{O}_{\boldsymbol{K}}),\text{det}^k)_{\mathbb{Z}_{(p)}}^{sym}$,
with even weight $k$ satisfies
$$
a(F;H) \equiv 0 \pmod{p}
$$
for all $H=\binom{m\,*}{*\,n}\in\Lambda_2(\mathcal{O}_{\boldsymbol{K}})$ with
$m,\,n\leq \left[\frac{k}{9}\right]$,
then
$$
F \equiv 0 \pmod{p}.
$$
\end{Thm}
\section{Main results}
As stated in the introduction, the main purpose of this paper is to constuct
some examples of Hermitian modular forms in the mod $p$ kernel of the
theta operator.
\subsection{Hermitian theta series}
\label{sec:4-1}
In $\S$  \ref{sec:2-2} , we considered the Hermitian theta series $\vartheta^{(2)}(Z,H_i)$
where $H_i$\,$(i=1,\ldots,5)$ are the Hermitian matrices corresponding to the representatives
of Eisenstein lattices of rank $12$. The first main result shows that two of
$\vartheta^{(2)}(Z,H_i)$ are in the mod $11$ kernel of the theta operator.
\begin{Thm}
\label{maintheta}
We have
$$
\varTheta(E_{12,\boldsymbol{K}}^{(2)})
\equiv \varTheta(\vartheta^{(2)}(Z,H_4)) \equiv \varTheta(\vartheta^{(2)}(Z,H_5))
\equiv 0 \pmod{11},
$$
in particular, $\vartheta^{(2)}(Z,H_4)$ and $\vartheta^{(2)}(Z,H_5)$
are in the mod 11 kernel of the theta operator.
\end{Thm}
\begin{proof}
The statement that $\varTheta (E_{12,\boldsymbol{K}}^{(2)}) \equiv 0 \pmod{11}$
is a consequence of \cite{K-N}, Theorem 5.
Next we shall show that
$$
\varTheta(\vartheta^{(2)}(Z,H_5)) \equiv 0 \pmod{11}.
$$
We apply the setting in Theorem \ref{B-N-Hvarsion} to
$$
\begin{cases}
 F=\vartheta^{(2)}(Z,H_5)
\in M_{12}(\Gamma_2(\mathcal{O}_{\boldsymbol{K}}),\text{det}^{-6})_{\mathbb{Z}}
=M_{12}(\Gamma_2(\mathcal{O}_{\boldsymbol{K}}),\text{det}^{12})_{\mathbb{Z}},\\
 p=11.
 \end{cases}
 $$
Then there is a modular form 
$G\in S_{24}(\Gamma_2(\mathcal{O}_{\boldsymbol{K}}),\text{det}^{24})_{\mathbb{Z}_{(11)}}$
such that
$$
\varTheta (\vartheta^{(2)}(Z,H_5)) \equiv G \pmod{11}.
$$
The numerical data (Table 1 in $\S$ \ref{sec:5-1}) on the Fourier coefficients $a(\vartheta^{(2)}(Z,H_5);H)$
shows that
$$
\det(H)\cdot a(G;H) \equiv 0 \pmod{11}
$$
for all $H=\binom{m\,*}{*\,n}$ with $n,\,m \leq\left[\frac{24}{9}\right]=2$.
By Sturm's bound (Theorem \ref{sturm}), we obtain
$$
\varTheta (\vartheta^{(2)}(Z,H_5)) \equiv G \equiv  0 \pmod{11}.
$$
Finally, we prove 
$\varTheta(\vartheta^{(2)}(Z,H_4)) \equiv 0 \pmod{11}$. We recall that
$\varTheta (E_{12,\boldsymbol{K}}^{(2}) \equiv 0 \pmod{11}$. 
However, it follows from Proposition \ref{siegel} that
$$
7\vartheta^{(2)}(Z,H_4)+5\vartheta^{(2)}(Z,H_5) \equiv
E_{12,\boldsymbol{K}}^{(2)} \pmod{11}.
$$
Since $\varTheta(\vartheta^{(2)}(Z,H_5)) \equiv \varTheta (E_{12,\boldsymbol{K}}^{(2)}) \equiv 0 \pmod{11}$, we obtain
$$
\varTheta(\vartheta^{(2)}(Z,H_4)) \equiv 0 \pmod{11}.
$$
\end{proof}
A congruence relation similar to the above Theorem holds if
$\boldsymbol{K}$ is the Gaussian field (Theorem 8 in \cite{K-N}).
These facts lead us the following conjecture:
\vspace{1mm}
\\
{\it Any Hermitian theta series associated with the Hermitian Leech lattice
will be in the
mod $p$ kernel of the theta operator for a suitable  prime number $p$.}

\subsection{Odd weight forms}
\label{sec:4-2}
In $\S$ \ref{sec:3-1}, we saw a set of generators of the graded ring
$\oplus M_k(\Gamma_2(\mathcal{O}_{\boldsymbol{K}}),\text{det}^{k})$ 
(cf. Theorem \ref{graded}). There are two modular forms with odd
weight in the set of generators. The second main result shows that these
forms are in the kernel of the theta operator.
\vspace{2mm}
\\
Let $\phi_9$ and $\phi_{45}$ be odd weight modular forms
given in Theorem 3.3. We assume that they are normalized as
$$
a\left( \phi_9;\begin{pmatrix}1 & -1/\sqrt{3}i \\ 1/\sqrt{3}i & 1
  \end{pmatrix}\right)
=
a\left( \phi_{45};\begin{pmatrix} 3 & 1 \\ 1 &  4
  \end{pmatrix}\right)
=1,
$$
(i.e., the {\it first} Fourier coefficient is equal to one).
\\
Now we use the following abbreviations for 
$H \in \Lambda_2(\mathcal{O}_{\boldsymbol{K}})$:
\begin{equation}
\label{abb}
H=
\begin{pmatrix} m & (a+b\sqrt{3}i)/2\sqrt{3}i \\
              (-a+b\sqrt{3}i)/2\sqrt{3}i  & n
\end{pmatrix}
=:(m,n,a,b)
\end{equation}
and
$$
q^{(m,n,a,b)}:=q^H=\text{exp}(2\pi i\text{tr}(HZ)).
$$
By numerical computation, a few Fourier coefficients of
$\phi_9$ are given as follows:
\begin{align*}
\phi_9=q^{(1,1,-2,0)}&-q^{(1,1,-1,-1)}-q^{(1,1,-1,1)}\\
  &+q^{(1,1,1,-1)}+q^{(1,1,1,1)}-q^{(1,1,2,0)}
+\sum_{H}a(\phi_9;H)q^H.
\end{align*}
In the last summation, $H$ runs over the elements of 
$\Lambda_2(\mathcal{O}_{\boldsymbol{K}})_{>0}$ such
that $H=\binom{m\, *}{*\, n}$ with $\text{max}(m,n)>1$.
\\
The coefficients of $\phi_{45}$ are given as follows:
\begin{align*}
\phi_{45}=& q^{(3,4,-3,-1)}-q^{(3,4,-3,1)}-q^{(3,4,0,2)}\\
&+q^{(3,4,3,-1)}-q^{(3,4,3,1)}-q^{(4,3,-3,-1)}+q^{(4,3,-3,1)}\\
&+q^{(4,3,0,-2)}-q^{(4,3,0,2)}-q^{(4,3,3,-1)}+q^{(4,3,3,1)}\\
&+\sum_{H}a(\phi_{45};H)q^H,
\end{align*}
where                          
$H$ runs over elements of 
$\Lambda_2(\mathcal{O}_{\boldsymbol{K}})_{>0}$ such
that $H=\binom{m\, *}{*\, n}$ with $\text{max}(m,n)>4$.
(Further examples of $a(\phi_9;H)$ and $a(\phi_{45};H)$ 
are given in $\S$ 5.)
\begin{Lem}
\label{Lem42}
We have
$$
\phi_9\in 
S_9(\Gamma_2(\mathcal{O}_{\boldsymbol{K}}),1)_{\mathbb{Z}}^{skew}
\quad \text{and}\quad
\phi_{45}\in
S_{45}(\Gamma_2(\mathcal{O}_{\boldsymbol{K}}),1)_{\mathbb{Z}}^{sym}.
$$
Moreover, let 
$\phi=\phi_9$ or $\phi_{45}$ and $f\in R[\![ q]\!]$ for some
ring $R\subset\mathbb{C}$, 
if there exists $g\in \mathbb{C}[\![ q]\!]$ such
that $f=\phi g$, then we have
$$
g\in R[\![ q]\!].
$$ 
\end{Lem}
\begin{proof}
The first statement of integrality follows from the fact that
$\phi_9$ and $\phi_{45}$ are Borcherds products.
The second statement follows from the above explicit Fourier
expansions of $\phi_9$ and $\phi_{45}$ because they are monic if
we define a suitable order of monomials.
\end{proof}
We denote by $\mathbb{S}_2$ the Siegel upper-half space of
degree 2. This is characterized by
$$
\mathbb{S}_2=\{\, Z\in \mathbb{H}_2\,\mid\, {}^tZ=Z\, \}.
$$
\begin{Lem}
\label{Lem43}
Let $p\geq 3$ be a prime number and
$F\in 
M_k(\Gamma_2(\mathcal{O}_{\boldsymbol{K}}),\text{det}^k)
_{\mathbb{Z}_{(p)}}^{sym}$ with $k$ odd. 
If $F|_{\mathbb{S}_2} \equiv 0 \pmod{p}$, then there exists
$$  G\in 
M_{k-18}(\Gamma_2(\mathcal{O}_{\boldsymbol{K}}),\text{det}^{k-18})
_{\mathbb{Z}_{(p)}}^{sym}
$$
such that $F \equiv \phi_9^2 G \pmod{p}$.
\end{Lem}
\begin{proof}
By Lemma 2 of \cite{D-K} and Lemma \ref{Lem42}, there exists
$$h\in M_{k-45}(\Gamma_2(\mathcal{O}_{\boldsymbol{K}}),\text{det}^{k-45})
_{\mathbb{Z}_{(p)}}^{sym}$$  
such that $F=\phi_{45}h$.
It is known that 
$\phi_{45}|_{\mathbb{S}_2}=\chi_{10}\chi_{35}$ where 
$\chi_{k}$ are Igusa's Siegel cusp forms of degree 2. 
Considering the Fourier expansion of 
$\chi_{10}\chi_{35}$, we see that $h|_{\mathbb{S}_2} \equiv 0 \pmod{p}$.
Since there exists
$J\in 
M_{k-45}(\Gamma_2(\mathcal{O}_{\boldsymbol{K}}),\text{det}^{k-45})
_{\mathbb{Z}_{(p)}}^{sym}$
such that $p\,J|_{\mathbb{S}_2}=h|_{\mathbb{S}_2}$, we have
$$
(F-p\,\phi_{45}J)|_{\mathbb{S}_2}=0.
$$
By \cite{D-K} and Lemma \ref{Lem42}, there exists
$$
G' \in 
M_{k-9}(\Gamma_2(\mathcal{O}_{\boldsymbol{K}}),\text{det}^{k-9})
_{\mathbb{Z}_{(p)}}^{skew}
$$
such that $F \equiv \phi_9G' \pmod{p}$. Again by Lemma 2 of \cite{D-K}
and Lemma \ref{Lem42}, there is
$G \in 
M_{k-18}(\Gamma_2(\mathcal{O}_{\boldsymbol{K}}),\text{det}^{k-18})
_{\mathbb{Z}_{(p)}}^{sym}$ such that
$G'=\phi_9G$.
\end{proof}
\begin{Lem}
\label{Lem44}
Let
$\displaystyle F=\sum_{m,n \geq 0}a_{m,n}(F;\dot{q}_{12},\ddot{q}_{12})q_{11}^m q_{22}^{n}
\in \mathbb{Z}_{(p)}[\![ q]\!]$ (cf. (\ref{power})).
Suppose that there is $G\in \mathbb{Z}_{(p)}[\![ q]\!]$ such that
$F \equiv \phi_9 G \pmod{p}$ and
$$
a_{m,n}(F;\dot{q}_{12},\ddot{q}_{12}) \equiv 0 \pmod{p} 
$$
for all $m,n \leq N$. Then we have
$$
a_{m,n}(G;\dot{q}_{12},\ddot{q}_{12}) \equiv 0 \pmod{p} 
$$
for all $m,n\leq N-1$.
\end{Lem}
\begin{proof}
As an element of $\mathbb{Z}_{(p)}[\![ q]\!]$, $\phi_9$ is
given by
\begin{align*}
\phi_9 &= \dot{q}_{11}\ddot{q}_{22}
(\dot{q}_{12}^{-2}-\dot{q}_{12}^{-1}\ddot{q}_{12}^{-1}
-\dot{q}_{12}^{-1}\ddot{q}_{12}
+\dot{q}_{12}\ddot{q}_{12}^{-1}
+\dot{q}_{12}\ddot{q}_{12}
-\ddot{q}_{12}^{2})\\ 
& +\sum_{\text{max}(m,n)>1}a(\phi_9;\dot{q}_{12},\ddot{q}_{12})
\dot{q}_{11}^m\dot{q}_{22}^n,\qquad\text{(cf. (\ref{power}))}.
\end{align*}
The statement follows from this expression.
\end{proof}
\begin{Prop}
\label{Prop45}
Let $p \geq 3$ be a prime number. Suppose that $k \geq 0$ is odd
and 
$F \in 
M_{k}(\Gamma_2(\mathcal{O}_{\boldsymbol{K}}),\text{det}^{k})
_{\mathbb{Z}_{(p)}}$.
If
$$
a(F;H) \equiv 0 \pmod{p}
$$
for any 
$H=\binom{m\, *}{*\, n}\in \Lambda_2(\mathcal{O}_{\boldsymbol{K}})_{\geq 0}$
such that $m,n\leq [k/9]-1$. Then we have $a(F;H) \equiv 0 \pmod{p}$ for any
$H \in \Lambda(\mathcal{O}_{\boldsymbol{K}})_{\geq 0}$, namely, 
$F \equiv 0 \pmod{p}$.
\end{Prop}
\begin{proof}
Since the proof is similar, we prove the case only when $F$ is {\it symmetric}.\\
We prove the statement by induction on $k$.

First, let $k=45$ and $F$ be a constant multiple of $\phi_{45}$.
(Note that $\phi_9$ is skew symmetric.) 
Then the statement follows from the Fourier expansion of $\phi_{45}$.
Next we assume that $k>45$ and the statement is true for smaller weights.
By assumption and Sturm bound for the Siegel modular case (cf. \cite{K-T}),
we have $F|_{\mathbb{S}_2} \equiv 0 \pmod{p}$. By Lemma \ref{Lem43},
there exists
$G \in M_{k-18}(\Gamma_2(\mathcal{O}_{\boldsymbol{K}}),\text{det}^{k-18})
_{\mathbb{Z}_{(p)}}^{sym}$ such that $F \equiv \phi_9^2G \pmod{p}$.
By Lemma \ref{Lem44}, for any
$H=\binom{m\, *}{*\, n}\in \Lambda(\mathcal{O}_{\boldsymbol{K}})_{\geq 0}$,
we have $a(G;H) \equiv 0 \pmod{p}$ if $m,n\leq [k/3]-3=[(k-18)/9]-1$.
By induction hypothesis, we have $G \equiv 0 \pmod{p}$.
\end{proof}
The second main result can be stated as follows:
\begin{Thm}
\label{mainodd}
We have the following congruence relations:
\vspace{2mm}
\\
(1)\; $\varTheta (\phi_9) \equiv 0 \pmod{2}$
\vspace{1mm}
\\
(2)\; $\varTheta (\phi_{45}) \equiv 0 \pmod{11}$.
\end{Thm}
\begin{proof}
(1)\;
We consider the Hermitian Rankin-Cohen bracket $[\phi_9,E_{4,\boldsymbol{K}}^{(2)}]$,
where $\phi_9$ and $E_{4,\boldsymbol{K}}^{(2)}$ are forms in the generators of the graded ring
$\oplus M_k(\Gamma_2(\mathcal{O}_{\boldsymbol{K}}),\text{det}^k)$. 
The Hermitian Rankin-Cohen bracket $[\,,\,]$ defined similarly as in the case of
Siegel modular forms (cf. \cite{B-N}, \cite{M-S}).
Using the fact that $E_{4,\boldsymbol{K}}^{(2)}\equiv 1 \pmod{2^4}$ and explicit expression of
$[\phi_9,E_{4,\boldsymbol{K}}^{(2)}]$, we have the following expression:
$$
[\phi_9,E_{4,\boldsymbol{K}}^{(2)}]=2^3c_0\cdot\varTheta(\phi_9)E_{4,\boldsymbol{K}}^{(2)}
+2^mc_1\cdot P
$$
where $c_i\in\mathbb{Z}_{(2)}$\; $((c_0,2)=1)$, $4\leq m\in\mathbb{Z}$, and $P$ is a Fouirer
series in $\mathbb{Z}[\![q]\!]$. We set $G_1:=(2^3c_0)^{-1}[\phi_9,E_{4,\boldsymbol{K}}^{(2)}]$.
Then we obtain 
$G_1\in S_{15}(\Gamma_2(\mathcal{O}_{\boldsymbol{K}}),\text{det}^{15})_{\mathbb{Z}_{(2)}}$
and 
$$
\varTheta (\phi_9) \equiv G_1 \pmod{2}.
$$
By Theorem \ref{graded}, we can write
$$
G_1=\gamma\cdot E_{6,\boldsymbol{K}}^{(2)}\,\phi_9\quad (\gamma\in\mathbb{Z}_{(2)}).
$$
Computing the Fourier coefficients of $\varTheta (\phi_9)$ and $G_1$ at
$H=\begin{pmatrix}1 & i/\sqrt{3} \\ -i/\sqrt{3} & 1\end{pmatrix}$, we obtain
$$
\begin{cases}
a(\varTheta (\phi_9);H)=\frac{2}{3} \equiv 0 \pmod{2},\\
a(G_1;H)=\gamma
\end{cases}
$$
This implies $\gamma \equiv 0\pmod{2}$, and we get
$$
\varTheta (\phi_9) \equiv G_1 \equiv 0 \pmod{2}.
$$
(2)\;
By Theorem \ref{B-N-Hvarsion}, there is a modular form
$G_2\in S_{57}(\Gamma_(\mathcal{O}_{\boldsymbol{K}}),\text{det}^{57})_{\mathbb{Z}_{(11)}}$
such that
$$
\varTheta (\phi_{45}) \equiv G_2 \pmod{11}.
$$
By Table 3 in $\S$ 5, we see that
$$
a(G_2;H) \equiv 0 \pmod{11}
$$ 
for all $H=\binom{m\,*}{*\,n}\in\Lambda_2(\mathcal{O}_{\boldsymbol{K}})_{\geq 0}$
with $m,\,n\leq [57/9]-1=5$. It follows from Proposition \ref{Prop45} that
$$
a(G_2;H) \equiv 0 \pmod{11}
$$
for all $H\in\Lambda_2(\mathcal{O}_{\boldsymbol{K}})$. This implies
$$
\varTheta (\phi_{45}) \equiv G_2 \equiv 0 \pmod{11}.
$$
\end{proof}
Finally, we refer to the mod $p$ property of Borcherds product.
In the case of Siegel modular forms, 
we know that Igusa's cusp form $\chi_{35}$, which is a typical example of
Borcherds product, represents an element in the mod 23 kernel of the theta operator. 
The above results in Theorem \ref{mainodd} lead us to the following conjecture.
\vspace{1mm}
\\
{\it Any modular form of odd weight coming from Borcherds product will be in the
mod $p$ kernel of the theta operator for a suitable  prime number $p$.}

\section{Tables}
In this section, we summarize the tables of Fourier coefficients that are needed
in the proof of statements in the previous sections.
\\
As in $\S$ 4, we use the following abbreviation:
$$
\Lambda_2(\mathcal{O}_{\boldsymbol{K}})\ni 
H=\begin{pmatrix}         m                                  & (a+b\sqrt{3}\,i)/(2\sqrt{3}\,i)\\
                      (-a+b\sqrt{3}\,i)/(2\sqrt{3}\,i)   &              n
    \end{pmatrix} 
    =:(m,n,a,b).
$$
\subsection{Fourier coefficients of Hermitian theta series}
\label{sec:5-1}
In $\S$ \ref{sec:4-1}, we considered the degree 2 theta series $\vartheta^{(2)}(Z,H_i)$ 
for rank 12 Eisenstein lattices $H_i$. The fifth matrix $H_5$ corresponds to the Hermitian
Leech lattice (cf. Proposition \ref{hermLeech}). The following table concerns the Fourier
coefficients $a(\vartheta^{(2)}(Z,H_5);H)$.
\begin{table*}[htbp]
\caption{Fourier coefficient $a(\vartheta^{(2)}(Z,H_5);H)$}
\begin{center}
\begin{tabular}{lll} \hline
$H$        &  3$\text{det}(H)$  &  $a(\vartheta^{(2)}(Z,H_5);H)$  \\ \hline
$(2,2,6,0)$&         3              &              0                         \\ \hline
$(2,2,5,1)$&         5              &              0                         \\ \hline
$(2,2,4,0)$& 8  &  $175134960=2^4\cdot 3^7\cdot 5\cdot 7\cdot 11\cdot 13$ \\ \hline
$(2,2,3,1)$& 9  &  $553512960=2^{12}\cdot 3^3\cdot 5\cdot 7\cdot 11\cdot 13$ \\ \hline
$(2,2,2,0)$& 11  &  $4075868160=2^{12}\cdot 3^7\cdot 5\cdot 7\cdot 13$ \\ \hline
$(2,2,0,0)$& 12  &  $980755760=2^7\cdot 3^7\cdot 5\cdot 7^2\cdot 11\cdot 13$ \\ \hline                                  
\end{tabular}
\end{center}
\end{table*}
\noindent
\subsection{Fourier coefficients of odd weight forms}
\label{sec:5-2}
Let $\phi_9$ be the modular form given in Proposition \ref{odd} which is the first
odd weight generator of the graded ring. We take a normalization 
$$
a(\phi_9;(1,1,-2,0))=1
$$
as in $\S$ 4.
\vspace{2mm}
\\
For $n\in \mathbb{Z}_{\geq 1}$, we set
$$
K_n:=
(1,n,2,0)\in\Lambda_2(\mathcal{O}_{\boldsymbol{K}}).
$$
Any non-zero Fourier coefficient $a(\phi_9;H)$ coincides with $a(\phi_9;K_n)$ for $K_n$
with $\text{det}(H)=\text{det}(K_n)$ up to sign.
\newpage
\noindent
\begin{table*}[hbtp]
\caption{Fourier coefficients $a(\phi_{9};K_n)$ for $n\leq 20$}
\begin{center}
\begin{tabular}{lll} \hline
$n$  &   3$\text{det}(K_n)$ & $a(\phi_{9};K_n)$ \\ \hline
1   &   2      &        -1            \\ \hline
2 &   5      &       $16=2^4$            \\ \hline
3  &  8      &        $-104=-2^3\cdot 13$  \\ \hline
4 &   11     &       $320=-2^6\cdot 5$         \\ \hline
5 &  14     &        $-260=-2^2\cdot 5\cdot 13$     \\ \hline
6 &  17     &        $-1248=-2^5\cdot 3\cdot 13$     \\ \hline
7 &  20     &       $ 3712=2^7\cdot 29$  \\ \hline
8 &  23     &       $-1664=-2^7\cdot 13$  \\ \hline
9 &  26      &       $ -6890=-2\cdot 5\cdot 13\cdot 53$ \\ \hline
10 & 29     &        $7280=2^4\cdot 5\cdot 7\cdot 13$  \\ \hline
11 & 32     &        $5568=2^6\cdot 3\cdot 29$            \\ \hline
12 & 35      &       $4160=2^6\cdot 5\cdot 13$            \\ \hline
13  & 38      &        $-33176=-2^3\cdot 11\cdot 13\cdot 29$  \\ \hline
14 &  41     &       $-4640=-2^5\cdot 5\cdot 29$         \\ \hline
15 &  44     &        $74240=2^9\cdot 5\cdot 29$     \\ \hline
16 &  47     &        $-29824=-2^7\cdot 233$     \\ \hline
17 &  50     &       $ -14035=-5\cdot 7\cdot 401$  \\ \hline
18 &  53     &       $-54288=-2^4\cdot 3^2\cdot 13\cdot 29$  \\ \hline
19 &  56      &      $-27040=-2^5\cdot 5\cdot 13^2$ \\ \hline
20 &  59     &        $142720=2^7\cdot 5\cdot 223$  \\ \hline
\end{tabular}
\end{center}
\end{table*}

\begin{Rem}
From Table 2, we can see that
the prime numbers 13 and 29 appear frequently
as the prime factors of $a(\phi_9;K_n)$. 
It has been confirmed that the phenomenon occurs for a wide range of $n$.
\end{Rem}
\newpage
\noindent
The second odd weight form in the set of generators is $\phi_{45}$.
\begin{table*}[hbtp]
\caption{Fourier coefficients $a(\phi_{45};H)$}
\begin{center}
\begin{tabular}{lll} \hline
$H$  &   3$\text{det}(H)$ & $a(\phi_{45};H)$ \\ \hline
(3,4,0,2)   &   33      &        1            \\ \hline
(3,5,0,-2) &   42       &       $88=2^3\cdot 11$            \\ \hline
(3,6,0,2)  &    51      &        $3740=2^2\cdot 5\cdot 11\cdot 17$  \\ \hline
(4,5,-5,1)&   53      &       $16038=2\cdot 3^6\cdot 11$         \\ \hline
(4,5,0,-2) &  57      &        $95931=3^3\cdot 11\cdot 17\cdot 19$     \\ \hline
(4,6,-5,1) &  65      &        $681615=3^6\cdot 5\cdot 11\cdot 17$     \\ \hline
(4,6,0,2)  &   69      &       $ 720940=2^2\cdot 5\cdot 11\cdot 29\cdot 113$  \\ \hline
(5,6,-7,-1)&  77      &       $ 47271276=2^2\cdot 3^6\cdot 13\cdot 29\cdot 43$  \\ \hline
(5,6,-6,-2)&  78      &       $ 13709344=2^5\cdot 11\cdot 17\cdot 29\cdot 79$ \\ \hline
(5,6,-5,1) &   83     &        $62772732=2^2\cdot 3^6\cdot 11\cdot 19\cdot 103$  \\ \hline
(5,6,-3,-1) & 87      &        $835953624=2^3\cdot 3^2\cdot 11\cdot 127\cdot 8311$            \\ \hline
\end{tabular}
\end{center}
\end{table*}




Shoyu Nagaoka\\
Dept. Mathematics Kindai Univ.
\\
Higashi-Osaka, Osaka 577-8502, Japan
\\
\\
Sho Takemori\\
Max Planck Institut f\"{u}r Mathematik,
\\
Vivatsgasse 7, 53111 Bonn, Germany

\end{document}